\documentclass[reqno]{amsart}

\usepackage{amssymb,amsfonts,amstext,amsthm,hyperref,cleveref,xcolor}
\usepackage{cite}

\newtheorem{thrm}{Theorem}[section]
\newtheorem{cor}[thrm]{Corollary}
\newtheorem{lem}[thrm]{Lemma}
\newtheorem{prop}[thrm]{Proposition}

\theoremstyle{definition}

\newtheorem{rem}[thrm]{Remark}

\crefrangeformat{equation}{#3(#1)#4--#5(#2)#6}

\crefname{thrm}{Theorem}{Theorems}
\crefname{lem}{Lemma}{Lemmas}
\crefname{cor}{Corollary}{Corollaries}
\crefname{prop}{Proposition}{Propositions}
\crefname{defn}{Definition}{Definitions}
\crefname{exm}{Example}{Examples}
\crefname{rem}{Remark}{Remarks}
\crefname{section}{Section}{Sections}
\crefname{equation}{\unskip}{\unskip}
\crefname{enumi}{\unskip}{\unskip}

\newcommand{\I}{I(X,R,\sigma)}

\begin{document}

\title[Skew Incidence Rings and the Isomorphism Problem]{Skew Incidence Rings and the Isomorphism Problem}

\author{\'Erica Z. Fornaroli}
\address{Departamento de Matem\'atica, Universidade Estadual de Maring\'a, Maring\'a--PR, CEP: 87020--900, Brazil}
\email{ezancanella@uem.br}

\begin{abstract}
	Let $X$ be a finite partially ordered set, $R$ an associative unital ring and $\sigma$ an endomorphism of $R$. We describe some properties of the skew incidence ring $\I$ such as invertible elements, idempotents, the Jacobson radical and the center. Moreover, if the skew incidence rings $\I$ and $I(Y,S,\tau)$ are isomorphic and the only idempotents of $R,S$ are the trivial ones, we show that the partially ordered sets $X$ and $Y$ are isomorphic.
\end{abstract}



\maketitle

\section*{Introduction}\label{intro}

Let $(X,\le)$ be a finite partially ordered set (finite poset, for short). It is known that $X$ can be labeled $X=\{x_1,\ldots, x_n\}$ such that $x_i\leq x_j$ implies $i\leq j$ (see ~\cite[Lemma 1.2.5]{SpDo}). If $x,y\in X$, $x\le y$ and $x\neq y$, we will just write $x<y$. Let $R$ be an associative ring with identity $1$ and let $\sigma: R\to R$ be an endomorphism such that $\sigma(1)=1$. The \emph{skew incidence ring $\I$ of $X$ over $R$ with respect to $\sigma$} is the set of functions $f: X\times X\to R$, such that $f(x,y)=0$ if $x\not\le y$, with the natural structure of a left $R$-module and the product defined by
\begin{align}\label{convolution}
(fg)(x_i,x_j)&=\sum_{x_i\le x_k\le x_j}f(x_i,x_k)\sigma^{k-i}(g(x_k,x_j)),
\end{align}
for any $f,g\in \I$ and $x_i\le x_j$ in $X$. The ring $\I$ is associative with identity $\delta$ given by $\delta(x,y)=1$ if $x=y$ and $\delta(x,y)=0$ if $x\neq y$ (see \cite{ezF}). Note that if $\sigma=Id_R$ is the identity endomorphism, then $\I=I(X,R)$ is the incidence ring of $X$ over $R$, and if $\sigma$ is an automorphism, then $I(X,R)\cong\I$ via $f\mapsto h$ where $h(x_i,x_j)=\sigma^{1-i}(f(x_i,x_j))$. In the case that $X=\{1,\ldots,n\}$ with the usual order, then $\I=UT_n(R,\sigma)$ is the skew triangular matrix ring defined by Chen et al. ~\cite{CYZ}. Note that $UT_2(R,\sigma)$ coincides with the formal triangular matrix ring $\begin{pmatrix} R & M\\ 0 & R\end{pmatrix}$ where ${}_RM={}_RR$ with $mr = m\sigma(r)$ for $m \in M$, $r\in R$. Skew triangular matrix rings have been investigated by many authors (see \cite{HM,HMA,NM,N1,N2,P}, for instance).

In \cref{sec-properties}, we describe the units of $\I$ and as a consequence we obtain its Jacobson radical. We also describe the idempotents and primitive idempotents of $\I$. Such results are applied in \cref{sec-isomorphism} to obtain a positive answer for the isomorphism problem: if the skew incidence rings $\I$ and $I(Y,S,\tau)$ are isomorphic and the only idempotents of $R,S$ are the trivial ones, then the partially ordered sets $X$ and $Y$ are isomorphic. Such result was proved by Voss in \cite{Voss} in case $R=S$, $\sigma=\tau=Id_R$ and $X,Y$ locally finite sets. Further references to the isomorphism problem for incidence rings and their generalizations are \cite{AHD,Khripchenko-Novikov09,Khripchenko10,SpDo}. We also provide a sufficient condition for skew incidence rings to be isomorphic.

\section{Properties of skew incidence rings}\label{sec-properties}

From now on, $R$ is an associative ring with identity, $\sigma: R\to R$ is an endomorphism and $X=\{x_1,\ldots, x_n\}$ is a poset such that $x_i\leq x_j$ implies $i\leq j$.

The equalities \cref{in_diagonal,af_xye_xy-cdot-bt_uve_uv,e_xye_zw,e_x-f-e_y} below are from \cite{ezF} and can be easily verified. Given $f,g\in \I$, we have
\begin{align}\label{in_diagonal}
(fg)(x,x)=f(x,x)g(x,x),
\end{align}
for all $x\in X$.

Since $X$ is finite, $\I$ is a free left $R$-module spanned by the set $\{e_{xy} : x\leq y\}$, where $e_{xy}(x,y)=1$ and $e_{xy}(u,v)=0$ if
$(u,v)\neq(x,y)$. We will also denote $e_{x}=e_{xx}$. For all $r,s\in R$,
\begin{align}\label{af_xye_xy-cdot-bt_uve_uv}
(re_{x_ix_j})(se_{x_kx_l})=
 \begin{cases}
  r\sigma^{j-i}(s)e_{x_ix_l}, & \mbox{if $j=k$},\\
  0, & \mbox{otherwise}.
 \end{cases}
\end{align}
In particular,
\begin{align}\label{e_xye_zw}
e_{xy}e_{zw}=\delta_{yz}e_{xw},
\end{align}
where $\delta_{yz}$ is the Kronecker delta. It follows that the elements $e_x$, $x\in X$, are pairwise orthogonal idempotents of $\I$.

We also have, for any $f\in \I$,
\begin{align}\label{e_x-f-e_y}
 e_xf e_y=\begin{cases}
          f(x,y)e_{xy}, & \mbox{ if }x\le y,\\
          0, & \mbox{ otherwise}.
         \end{cases}
\end{align}

Given $x,y\in X$, we recall that the \emph{interval} from $x$ to $y$ is the set $\{z\in X : x\leq z\leq y\}$ and is denoted by $[x,y]$. An interval $[x,y]$ is said to have \emph{length} $m$ if there is a chain with $m$ elements in $[x,y]$, and any chain in $[x,y]$ has at most $m$ elements. The length of $[x,y]$ will be denoted by $|[x,y]|$.

\begin{prop}\label{invertible}
An element $f\in \I$ is invertible if, and only if, $f(x,x)$ is invertible in $R$ for all $x\in X$.
\end{prop}
\begin{proof}
If there is $g\in \I$ such that $fg=gf=\delta$, then by \eqref{in_diagonal},
$$f(x,x)g(x,x)=(fg)(x,x)=\delta(x,x)=1=\delta(x,x)=(gf)(x,x)=g(x,x)f(x,x),$$
for all $x\in X$.

On the other hand, suppose that $f(x,x)$ is invertible in $R$ for all $x\in X$. Note that if there is $g\in \I$ satisfying
\begin{enumerate}
\item [(i)] $g(x_i,x_i)=f(x_i,x_i)^{-1}$ for all $x_i\in X$;
\item [(ii)] $g(x_i,x_j)=-f(x_i,x_i)^{-1}\displaystyle\sum_{x_i< x_k\le x_j}f(x_i,x_k)\sigma^{k-i}(g(x_k,x_j))$ for all $x_i<x_j$ in $X$,
\end{enumerate}
then $g$ will be a right inverse of $f$, by \cref{convolution,in_diagonal}. We define such $g\in \I$ inductively on the length of the intervals of $X$ as follows. Let $x_i,x_j\in X$. If $|[x_i,x_j]|=0$, then $x_i\not\le x_j$ and we define $g(x_i,x_j)=0$. If $|[x_i,x_j]|=1$, then $x_i=x_j$ and we define $g(x_i,x_i)=[f(x_i,x_i)]^{-1}$. Suppose that $|[x_i,x_j]|=m>1$ and that for each interval of length less than $m$ the function $g$ has been defined on that interval. For each $x_k\in X$ such that $x_i<x_k\leq x_j$, we have $|[x_k,x_j]|<|[x_i,x_j]|=m$ and, therefore, $g(x_k,x_j)$ has already been defined. Thus, all summands on the right hand side of (ii) has been defined and then $g(x_i,x_j)$ is defined.

Analogously, we can define $h\in \I$ satisfying
\begin{enumerate}
\item [(iii)] $h(x_i,x_i)=f(x_i,x_i)^{-1}$ for all $x_i\in X$;
\item [(iv)] $h(x_i,x_j)=\left[-\displaystyle\sum_{x_i\le x_k< x_j}h(x_i,x_k)\sigma^{k-i}(f(x_k,x_j))\right]\sigma^{j-i}(f(x_j,x_j))^{-1}$ for all $x_i<x_j$ in $X$,
\end{enumerate}
and such $h$ will be a left inverse of $f$. Thus $h=g$ and, therefore, $fg=gf=\delta$.
\end{proof}

For any ring $S$, we will denote the Jacobson radical of $S$ by $J(S)$.

\begin{cor}
$J(\I)=\{f\in \I :  f(x,x)\in J(R) \text{ for all } x\in X\}$.
\end{cor}
\begin{proof}
Let $f\in J(\I)$. Then $\delta-gfh$ is invertible in $\I$ for all $g,h\in \I$, that is, $1-g(x,x)f(x,x)h(x,x)$ is invertible in $R$ for all $x\in X$, for all $g,h\in \I$, by \cref{invertible} and equality \cref{in_diagonal}. Let $a,b\in R$ and consider $g=a\delta$ and $h=b\delta$. Then $1-af(x,x)b=1-g(x,x)f(x,x)h(x,x)$ is invertible in $R$ for all $x\in X$ and any $a,b\in R$. Therefore, $f(x,x)\in J(R)$ for all $x\in X$.

On the other hand, let $f\in\I$ such that $f(x,x)\in J(R)$ for all $x\in X$. Then, given $g,h\in \I$, the element $(\delta-gfh)(x,x)=1-g(x,x)f(x,x)h(x,x)$ of $R$ is invertible for all $x\in X$. By \cref{invertible}, $\delta-gfh$ is invertible in $\I$ for all $g,h\in \I$ and, therefore, $f\in J(\I)$.
\end{proof}

The following results describe the idempotents and primitive idempotents of $\I$.

An element $f\in \I$ will be called \emph{diagonal} if $f(x,y)=0$ for $x\neq y$.

\begin{rem}\label{diagonal_idempotent}
Let $f\in \I$. If $f$ is idempotent, then $f(x,x)$ is idempotent for all $x\in X$, by \cref{in_diagonal}. On the other hand, if $f$ is diagonal and $f(x,x)$ is idempotent for all $x\in X$, then $f$ is idempotent, by \cref{convolution,in_diagonal}.
\end{rem}

\begin{prop}\label{idempotent}
Each idempotent $f\in \I$ is conjugate to the diagonal idempotent $e$, such that $e(x,x)=f(x,x)$ for all $x\in X$.
\end{prop}
\begin{proof}
Let $g=f-e$ and $h=\delta+(2e-\delta)g$. If $x\in X$, then $g(x,x)=0$ and, therefore, $h(x,x)=1$. Thus $h$ is invertible, by \cref{invertible}. We have
\begin{align*}
hf=& [\delta+(2e-\delta)g]f=f+(2e-\delta)(f-e)f=f+(2e-\delta)(f-ef)\\
  =& f+2ef-2ef-f+ef=ef
\end{align*}
and
$$eh= e[\delta+(2e-\delta)g]=e+2eg-eg=e+eg=e(e+g)=ef.$$
So, $hf=eh$ and then $f=h^{-1}eh$.
\end{proof}

We recall that an idempotent $e\neq 0$ is \emph{primitive}, if $ef=fe=f$ for some idempotent $f$ implies that $f=0$ or $f=e$.

\begin{prop}\label{primitive}
An idempotent $f\in\I$ is primitive if, and only if, it is conjugate to $ae_x$ for some primitive idempotent $a\in R$ and some $x\in X$. Moreover, such $x$ is unique. If $R$ is commutative or if the only idempotents of $R$ are $0$ and $1$, then $a$ is also unique.
\end{prop}
\begin{proof}
If $f\in \I$ is a primitive idempotent, then $f$ is conjugate to a diagonal idempotent which is also primitive, since the conjugation is an automorphism. Therefore, it is sufficient to prove that the diagonal primitive idempotents of $\I$ are exactly the $ae_x$ with $x\in X$ and $a$ a primitive idempotent of $R$.

Let $a$ be a primitive idempotent of $R$ and let $x\in X$. Then $ae_x$ is a diagonal idempotent of $\I$, by \cref{diagonal_idempotent}. Let $e\in \I$ be an idempotent such that $e(ae_x)=(ae_x)e=e$. For $x_i\leq x_j$ in $X$ we have
\begin{align*}
e(x_i,x_j)=& (e(ae_x))(x_i,x_j)=\sum_{x_i\leq x_k\leq x_j}e(x_i,x_k)\sigma^{k-i}(ae_x(x_k,x_j))\\
  =& \begin{cases}
          e(x_i,x)\sigma^{j-i}(a), & \mbox{ if }x_j=x,\\
          0, & \mbox{ if }x_j\neq x
         \end{cases}
\end{align*}
and
\begin{align*}
e(x_i,x_j)=& ((ae_x)e)(x_i,x_j)=\sum_{x_i\leq x_k\leq x_j}ae_x(x_i,x_k)\sigma^{k-i}(e(x_k,x_j))\\
  =& \begin{cases}
          ae(x,x_j), & \mbox{ if }x_i=x,\\
          0, & \mbox{ if }x_i\neq x.
         \end{cases}
\end{align*}
Thus, $e(x_i,x_j)=0$ if $(x_i,x_j)\neq (x,x)$ and $e(x,x)=e(x,x)a=ae(x,x)$. Since $e(x,x)$ is an idempotent and $a$ is a primitive idempotent of $R$, then $e(x,x)=0$ or $e(x,x)=a$. Therefore, $e=0$ or $e=ae_x$ and so $ae_x$ is a primitive idempotent of $\I$.

Let $g\in\I$ be a diagonal primitive idempotent. Since $g\neq 0$, there is $x\in X$ such that $g(x,x)\neq 0$. Let $b\in R$ be an idempotent such that $g(x,x)b=bg(x,x)=b$. Since $g$ and $be_x$ are diagonal, then $g(be_x)$ and $(be_x)g$ are also diagonal. Moreover, $be_x$ is an idempotent, by \cref{diagonal_idempotent}. By \cref{in_diagonal},
$$(g(be_x))(x,x)=((be_x)g)(x,x)=b=(be_x)(x,x)$$
and for $u\neq x$,
$$(g(be_x))(u,u)=((be_x)g)(u,u)=0=(be_x)(u,u).$$
Thus, $g(be_x)=(be_x)g=be_x$ and, therefore, $be_x=0$ or $be_x=g$. So, $b=0$ or $b=g(x,x)$ and then $g(x,x)$ is a primitive idempotent of $R$. Consider the idempotent $g(x,x)e_x\in\I$. Then $g(g(x,x)e_x)=(g(x,x)e_x)g=g(x,x)e_x$, because $g$ and $g(x,x)e_x$ are diagonal and $g(x,x)$ is an idempotent. Since the idempotent $g$ is primitive and $g(x,x)\neq 0$, then $g(x,x)e_x=g$.

Now, suppose that a primitive idempotent $f\in\I$ is conjugate to $ae_x$ and $be_y$, where $x,y\in X$ and $a,b$ are primitive idempotents of $R$. Let $g,h$ be invertible elements of $\I$ such that $f=g(ae_x)g^{-1}=h(be_y)h^{-1}$. If $x\neq y$, then
$$f(x,x)=(h(be_y)h^{-1})(x,x)=h(x,x)be_y(x,x)h^{-1}(x,x)=0.$$
But, on the other hand,
$$f(x,x)=(g(ae_x)g^{-1})(x,x)=g(x,x)ae_x(x,x)g^{-1}(x,x)=g(x,x)ag(x,x)^{-1}.$$
Thus, $g(x,x)ag(x,x)^{-1}=0$ and, therefore, $a=0$, which is a contradiction. It follows that $f=g(ae_x)g^{-1}=h(be_x)h^{-1}$ and then
$$g(x,x)ag(x,x)^{-1}=h(x,x)bh(x,x)^{-1}.$$
Therefore, if $R$ is commutative, then $a=b$. If the only idempotents of $R$ are $0$ and $1$, then $a=1=b$.
\end{proof}

We finish this section by describing the center of $\I$. For any ring $S$, we will denote the center of $S$ by $Z(S)$.

\begin{prop}\label{center}
Let $f\in\I$. Then $f\in Z(\I)$ if, and only if, the following statements are true:
\begin{enumerate}
\item $f$ is diagonal.
\item $f(x,x)\in Z(R)$ for all $x\in X$.
\item $f(x_i,x_i)=\sigma^{j-i}(f(x_j,x_j))$ for all $x_i\leq x_j$.
\end{enumerate}
\end{prop}
\begin{proof}
Let $f\in Z(\I)$. If $x_i<x_j$, then
$$(e_{x_i}f)(x_i,x_j)=\sum_{x_i\leq x_k\leq x_j}e_{x_i}(x_i,x_k)\sigma^{k-i}(f(x_k,x_j))=f(x_i,x_j)$$
and
$$(fe_{x_i})(x_i,x_j)=\sum_{x_i\leq x_k\leq x_j}f(x_i,x_k)\sigma^{k-i}(e_{x_i}(x_k,x_j))=0$$
Thus, $f(x_i,x_j)=0$ and, therefore, $f$ is diagonal. Given $a\in R$, consider $a\delta\in\I$. Since $((a\delta)f)(x,x)=af(x,x)$ and $(f(a\delta))(x,x)=f(x,x)a$, we have $af(x,x)=f(x,x)a$ and, therefore, $f(x,x)\in Z(R)$ for all $x\in X$. Let $x_i\leq x_j$. Then
$$(e_{x_ix_j}f)(x_i,x_j)=\sum_{x_i\leq x_k\leq x_j}e_{x_ix_j}(x_i,x_k)\sigma^{k-i}(f(x_k,x_j))=\sigma^{j-i}(f(x_j,x_j))$$
and
$$(fe_{x_ix_j})(x_i,x_j)=\sum_{x_i\leq x_k\leq x_j}f(x_i,x_k)\sigma^{k-i}(e_{x_ix_j}(x_k,x_j))=f(x_i,x_i),$$
therefore $f(x_i,x_i)=\sigma^{j-i}(f(x_j,x_j))$.

Conversely, let $f\in\I$ satisfying (i)-(iii) and let $g\in\I$ be an arbitrary element. For $x_i\leq x_j$ we have
\begin{align*}
(gf)(x_i,x_j)=& \sum_{x_i\leq x_k\leq x_j}g(x_i,x_k)\sigma^{k-i}(f(x_k,x_j))\\
             =& g(x_i,x_j)\sigma^{j-i}(f(x_j,x_j)) \text{ by (i)}\\
             =& g(x_i,x_j)f(x_i,x_i) \text{ by (iii)}\\
             =& f(x_i,x_i)g(x_i,x_j) \text{ by (ii)}\\
             =& \sum_{x_i\leq x_k\leq x_j}f(x_i,x_k)\sigma^{k-i}(g(x_k,x_j)) \text{ by (i)}\\
             =& (fg)(x_i,x_j).
\end{align*}
Therefore, $gf=fg$ and $f\in Z(\I)$.
\end{proof}

\begin{cor}
Let $f\in Z(\I)$ and let $x_i,x_j$ be elements in the same connected component of $X$. Then $\sigma^i(f(x_i,x_i))=\sigma^{j}(f(x_j,x_j))$.
\end{cor}
\begin{proof}
By the hypothesis, there are a positive integer $m$ and $x_i=u_0, u_1,\ldots, u_m=x_j$ elements of $X$ such that $u_l\leq u_{l+1}$ or $u_{l+i}\leq u_l$ for $l=0,\ldots,m-1$. We will prove that $\sigma^i(f(x_i,x_i))=\sigma^{j}(f(x_j,x_j))$ by induction on $m$. If $m=1$, then $x_i\leq x_j$ or $x_j\leq x_i$ and, by \cref{center}, $f(x_i,x_i)=\sigma^{j-i}(f(x_j,x_j))$ or $f(x_j,x_j)=\sigma^{i-j}(f(x_i,x_i))$. Therefore, $\sigma^i(f(x_i,x_i))=\sigma^{j}(f(x_j,x_j))$. Assume that $m\geq 2$ and the result is known for $m-1$. Thus, if $u_{m-1}=x_r$, we have $\sigma^i(f(x_i,x_i))=\sigma^{r}(f(x_r,x_r))$. Moreover, by the $m=1$ case, $\sigma^{r}(f(x_r,x_r))=\sigma^{j}(f(x_j,x_j))$. Therefore, $\sigma^i(f(x_i,x_i))=\sigma^{j}(f(x_j,x_j))$.
\end{proof}

\section{The isomorphism problem}\label{sec-isomorphism}

We recall that an \emph{isomorphism} from a poset $(P,\leq)$ onto a poset $(Q,\preceq)$ is a bijective map $\lambda: P\to Q$ that satisfies
$$x\leq y \Leftrightarrow \lambda(x)\preceq\lambda(y).$$

\begin{prop}
Let $(Y,\preceq)$ be a poset and let $S$ be a ring with an endomorphism $\tau$. Suppose there are an isomorphism of posets $\alpha:X\to Y$ and an isomorphism of rings $\varphi: R\to S$ such that $\varphi\sigma=\tau\varphi$. Let $y_i:=\alpha(x_i)$ for each $i=1,\ldots,n$. Then $\psi: \I\to I(Y,S,\tau)$ defined by
$$\psi(f)(y_i,y_j)=\varphi(f(x_i,x_j)),$$
for all $i,j=1,\ldots,n$, is an isomorphism of rings.
\end{prop}
\begin{proof}
Note that if $y_i,y_j\in Y$ and $y_i\npreceq y_j$, then $\alpha(x_i)\npreceq \alpha(x_j)$ and, therefore, $x_i\nleq x_j$. Thus, for each $f\in \I$, $\psi(f)(y_i,y_j)=\varphi(f(x_i,x_j))=\varphi(0)=0$ and so $\psi(f)\in I(Y,S,\tau)$.

Let $f,g\in \I$. It is easy to see that $\psi(f+g)=\psi(f)+\psi(g)$ and $\psi(\delta)=\delta$. Let $y_i\preceq y_j$ in $Y$. We have
\begin{align*}
\psi(fg)(y_i,y_j)=& \varphi((fg)(x_i,x_j))=\varphi\left(\sum_{x_i\leq x_k\leq x_j}f(x_i,x_k)\sigma^{k-i}(g(x_k,x_j))\right)\\
                 =& \sum_{x_i\leq x_k\leq x_j}\varphi(f(x_i,x_k))\varphi(\sigma^{k-i}(g(x_k,x_j)))\\
                 =& \sum_{x_i\leq x_k\leq x_j}\varphi(f(x_i,x_k))\tau^{k-i}(\varphi(g(x_k,x_j)))\\
                 =& \sum_{y_i\preceq y_k\preceq y_j}\psi(f)(y_i,y_k)\tau^{k-i}(\psi(g)(y_k,y_j))\\
                 =& (\psi(f)\psi(g))(y_i,y_j).
\end{align*}
Thus $\psi(fg)=\psi(f)\psi(g)$ and, therefore, $\psi$ is a homomorphism.

Now, let $\eta: I(Y,S,\tau)\to \I$ defined by $\eta(h)(x_i,x_j)=\varphi^{-1}(h(y_i,y_j))$, for all $i,j=1,\ldots,n$. We have
$\eta\circ \psi=Id_{\I}$ and $\psi\circ\eta=Id_{I(Y,S,\tau)}$. Thus, $\psi$ is an isomorphism with $\psi^{-1}=\eta$.
\end{proof}

To prove our main result, we need the following lemma.

\begin{lem}
Let $x,y\in X$. Then
$$x\leq y \Leftrightarrow e_x\I e_y\neq \{0\}.$$
\end{lem}
\begin{proof}
Suppose that $x\leq y$  and consider $e_{xy}\in \I$. By \cref{e_xye_zw}, $e_xe_{xy}e_y=e_{xy}$ and then $e_x\I e_y\neq \{0\}$. On the other hand, if $e_x\I e_y\neq \{0\}$, there is $f\in\I$ such that $e_xfe_y\neq \{0\}$. Thus, by \cref{e_x-f-e_y}, $x\leq y$.
\end{proof}

\begin{thrm}
Let $(Y,\preceq)$ be a poset and let $S$ be a ring with an endomorphism $\tau$. Suppose there is an isomorphism $\phi:\I\to I(Y,S,\tau)$.
\begin{enumerate}
\item If the only idempotents of $R$ and $S$ are the trivial ones, then $X\cong Y$.
\item If $\phi(R\delta)=S\delta$, then $R\cong S$.
\end{enumerate}
\end{thrm}
\begin{proof}
(i) Given $x\in X$ and $y\in Y$, we denote the elements $e_x\in \I$ and $e_y\in I(Y,S,\tau)$ by $e_x^X$ and $e_y^Y$, respectively. For each $x\in X$, $e_x^X$ is a primitive idempotent of $\I$, by \cref{primitive}. Thus $\phi(e_x^X)$ is a primitive idempotent of $I(Y,S,\tau)$ and, therefore, there is only one $y\in Y$ such that $\phi(e_x^X)$ is conjugate to $e_y^Y$, by \cref{primitive}. It follows that $\phi$ induces a map $\alpha: X\to Y$ such that, for each $x\in X$, $\phi(e_x^X)$ is conjugate to $e_{\alpha(x)}^Y$.

Let $x,u\in X$ such that $\alpha(x)=\alpha(u)$. If $f,g\in I(Y,S,\tau)$ are such that $\phi(e_x^X)=fe_{\alpha(x)}^Yf^{-1}$ and $\phi(e_u^X)=ge_{\alpha(u)}^Yg^{-1}$, then
\begin{align*}
e_x^X=& \phi^{-1}(f)\phi^{-1}(e_{\alpha(x)}^Y)\phi^{-1}(f)^{-1}=\phi^{-1}(f)\phi^{-1}(e_{\alpha(u)}^Y)\phi^{-1}(f)^{-1}\\
     =& \phi^{-1}(f)\phi^{-1}(g)^{-1}e_u^X\phi^{-1}(g)\phi^{-1}(f)^{-1}\\
     =& [\phi^{-1}(f)\phi^{-1}(g)^{-1}]e_u^X[\phi^{-1}(f)\phi^{-1}(g)^{-1}]^{-1}.
\end{align*}
By \cref{primitive}, $x=y$ and, therefore, $\alpha$ is injective. Now, given $y\in Y$, consider the primitive idempotent $e_y^Y\in I(Y,S,\tau)$. Then $\phi^{-1}(e_y^Y)$ is a primitive idempotent of $\I$. By \cref{primitive}, there is $x\in X$ such that $\phi^{-1}(e_y^Y)$ is conjugate to $e_x^X$. Thus $\phi(e_x^X)$ is conjugate to $e_y^Y$ and, therefore, $\alpha(x)=y$. So, $\alpha$ is surjective.

Finally, let $x,u\in X$ such that $x\leq u$, and let $f,g\in I(Y,S,\tau)$ such that $\phi(e_x^X)=fe_{\alpha(x)}^Yf^{-1}$ and $\phi(e_u^X)=ge_{\alpha(u)}^Yg^{-1}$. By lemma above, $e_x^X\I e_u^X\neq \{0\}$ and, therefore,
\begin{align*}
\{0\}\neq \phi(e_x^X\I e_u^X)=& \phi(e_x^X)\phi(\I)\phi(e_u^X)\\
                         =& fe_{\alpha(x)}^Yf^{-1}I(Y,S,\tau)ge_{\alpha(u)}^Yg^{-1}\\
                         =& fe_{\alpha(x)}^YI(Y,S,\tau)e_{\alpha(u)}^Yg^{-1}.
\end{align*}
Thus, $e_{\alpha(x)}^Y I(Y,S,\tau)e_{\alpha(u)}^Y\neq \{0\}$ and, therefore, $\alpha(x)\leq\alpha(u)$, by lemma above. Analogously, if $\alpha(x)\leq\alpha(u)$ then $x\leq u$.

(ii) Obvious.
\end{proof}

%\bibliography{bibl}{}

\begin{thebibliography}{99}

\bibitem{AHD}
G.~Abrams, J.~Haefner and \'A. del R\'io, Corrections and addenda to: ``The isomorphism problem for incidence rings",
\textit{Pacific J. Math.}~\textbf{207} 2 (2002), 497--506.

\bibitem{CYZ}
J.~Chen, X.~Yang and Y.~Zhou, On strongly clean matrix and triangular matrix rings,
\textit{Comm. Algebra}~\textbf{34} 10 (2006), 3659--3674.

\bibitem{ezF}
E.~Z.~Fornaroli, Jordan isomorphisms of skew incidence rings,
\textit{to appear in J. Algebra Appl.}.

\bibitem{HM}
M.~Habibi and A.~Moussavi, Special properties of a skew triangular matrix ring with constant diagonal,
\textit{Asian-Eur. J. Math.}~\textbf{8} 3 (2015), 1550021, 10pp.

\bibitem{HMA}
M.~Habibi, A.~Moussavi and A.~Alhevaz, On skew triangular matrix rings,
\textit{Algebra Colloq.}~\textbf{22} 2 (2015), 271--280.

\bibitem{Khripchenko-Novikov09}
N.~S.~Khripchenko and B.~V.~Novikov, Finitary incidence algebras,
\textit{Comm. Algebra}~\textbf{37} 5 (2009), 1670--1676.

\bibitem{Khripchenko10}
N.~S.~Khripchenko, Finitary incidence algebras of quasiorders,
\textit{Mat. Stud.}~\textbf{34} 1 (2010), 30--37.

\bibitem{NM}
A.~R.~Nasr-Isfahani and A.~Moussavi, On a quotient of polynomial rings,
\textit{Comm. Algebra}~\textbf{38} 2 (2010), 567--575.

\bibitem{N1}
A.~R.~Nasr-Isfahani, On skew triangular matrix rings,
\textit{Comm. Algebra}~\textbf{39} 11 (2011), 4461--4469.

\bibitem{N2}
A.~R.~Nasr-Isfahani, On a quotient of skew polynomial rings,
\textit{Comm. Algebra}~\textbf{41} 12 (2013), 4520--4533.

\bibitem{P}
K.~Paykan, Some new results on skew triangular matrix rings with constant diagonal,
\textit{Vietnam J. Math.}~\textbf{45} 4 (2017), 575--584.

\bibitem{SpDo}
E.~Spiegel and C.~J. O'Donnell, \textit{Incidence algebras}, Marcel Dekker, New York, 1997.

\bibitem{Voss}
E.~R.~Voss, On the isomorphism problem for incidence rings,
\textit{Illinois J. Math.}~\textbf{24} 4 (1980), 624--638.

\end{thebibliography}
%\bibliographystyle{acm}

\end{document}